\documentclass{article}       
\usepackage{graphicx}
\usepackage{enumitem}
\DeclareGraphicsExtensions{.bmp} 
\usepackage{subfigure}
\usepackage{epstopdf}
\usepackage{cite}
\usepackage[centertags]{amsmath}
\usepackage{amsfonts,amssymb,newlfont,makeidx,listings,xcolor}
\usepackage[colorlinks=true,urlcolor=blue,linkcolor=blue,pdfborder={0 0 0}]{hyperref}
\usepackage[left=20mm,top=20mm,right=20mm,bottom=20mm,twoside=false,marginparwidth=20mm,marginparsep=20mm]{geometry}

\newtheorem{theorem}{Theorem}[section]
\newtheorem{corollary}{Corollary}[section]
\newtheorem{remark}{Remark}[section]
\newtheorem{definition}{Definition}[section]
\newenvironment{proof}{\textit{Proof}:}{\hfill$\square$}
\numberwithin{equation}{section}

\begin{document}

	\title{A semi-symmetric metric connection, perfect fluid space-time and phantom barrier}
	\date{}
	\author{\textbf{Miroslav D. Maksimovi\'c\footnote{Corresponding: miroslav.maksimovic@pr.ac.rs}, Milan Lj. Zlatanovi\'c, Marija S. Najdanovi\'c}}
	\maketitle

	\begin{abstract}
		{\small We consider a concircularly semi-symmetric metric connection and its application. The Ricci tensors with respect to the concircularly semi-symmetric metric connection are symmetric, and they are used to define Einstein type manifolds. In this way, conditions under which a pseudo-Riemannian manifold is quasi-Einstein are obtained. On a Lorentzian manifold, a concircularly semi-symmetric metric connection with a unit timelike generator becomes a semi-symmetric metric $P$-connection, and a Lorentzian manifold becomes a GRW space-time. { It is proven in which case the scalar curvature of a perfect fluid space-time with that connection is constant and what its value is, which implies a reduction to Einstein manifold. } Furthermore, an application to the theory of relativity is presented, and the value of the equation of state is examined. It is ultimately shown that the equation of state in a  perfect fluid space-time that satisfies Einstein field equation with cosmological constant and admits a unit timelike torse-forming vector represents a phantom barrier.}
		
		\vskip0.25cm
		\noindent\textbf{Keywords}: Lorentzian manifold, semi-symmetric connection, space-time, perfect fluid, GRW space-time, phantom barrier.
		
		\vskip0.25cm
		\noindent\textbf{MSC 2020}: 53B05, 53B30, 53B50, 53C25, 83C05.
	\end{abstract}

	\section{Introduction and motivation}
	
	On a pseudo-Riemannian manifold $(\mathcal{M},g)$, an arbitrary non-zero vector $P$ is called \textit{timelike}, \textit{isotropic} (\textit{null}, \textit{lightlike}) and \textit{spacelike} if $g(P,P)<0$, $g(P,P)=0$ and $g(P,P)>0$, respectively. A zero vector $P=\vec{0}$ is spacelike.
	
	A \textit{torse-forming vector field} $P$ is a vector that satisfies the relation
	\begin{equation}\label{eq:torse-forming}
		\overset{g}{\nabla}_{X} P= \omega X + \eta(X)P,
	\end{equation}
	where $\eta$ is an arbitrary 1-form and $\omega$ is a scalar function \cite{yano1944}. {Using a torse-forming vector, almost geodesic mapping of the third type is defined (see pp. 168. in \cite{sinyukov1979})}. Depending on the properties of $\eta$ and $\omega$, we have the following special cases of the torse-forming vector $P$ (see \cite{chaubey2022}):
	\begin{enumerate}[label=\textup{(\roman*)}, itemsep=0.5em, leftmargin=*, font=\normalfont]
		\item \textit{torqued vector field} if $\eta(P)=0$,
		\item \textit{concircular vector field} (in Fialkow's sense) if $\eta=0$,
		\item \textit{concircular vector field} (in Yano's sense) if the 1-form $\eta$ is closed,
		\item \textit{recurrent vector field} if $\omega=0$,
		\item \textit{concurrent vector field} if $\eta=0$ and $\omega=1$,
		\item \textit{parallel vector field} if $\eta=0$ and $\omega=0$.
	\end{enumerate}

	Let $\pi$ be a 1-from associated with vector field $P$, i.e. $\pi(\cdot)=g(\cdot,P)$. If $P$ is a unit vector on a pseudo-Riemannian manifold, i.e. $g(P,P)=\epsilon=\pm 1$, then equation (\ref{eq:torse-forming}) takes the form (see \cite{yano1944})
	\begin{equation}\label{eq:torse-formingunit}
		(\overset{g}{\nabla}_{X} \pi )(Y)= \omega (g(X,Y) -\epsilon \pi(X)\pi(Y)).
	\end{equation}
	The case for a unit spacelike vector, i.e. $g(P,P)=1$, was studied in \cite{chaubey2022} and such a vector is concircular in Yano's sense. On the other hand, in \cite{deshmukh2021} such a vector is called an \textit{anti-torqued vector}, because the vectors $U$ and $P$ are parallel, i.e. $U=-\omega P$, where $U$ is an associated vector for the 1-form $\eta$, i.e. $\eta(\cdot)=g(\cdot,U)$. For $\omega\in\mathbb{R}$, in \cite{naik2021} such a vector is called \textit{concurent-recurrent}.
	
	A vector field $P$ is said to be \textit{conformal} (or \textit{conformal-Killing}) if it satisfies equation
	\begin{equation*}
		{\mathcal{L}}_{P} g = 2\varphi g,
	\end{equation*}
	where $\mathcal{L}$ denotes Lie derivative with respect to the Levi-Civita connection, and $\varphi$ is a scalar function. If $\varphi=0$ then $P$ is called \textit{Killing vector}. A vector field $P$ is said to be \textit{geodesic} if it satisfies equation 
	\begin{equation*}
		\overset{g}{\nabla}_{P} P=0.
	\end{equation*}
	
	If for a unit timelike vector $P$, i.e. $g(P,P)=-1$, we define a (1,1)-tensor $\phi X= \frac{1}{\omega} \overset{g}{\nabla}_{X} P$, then an $n$-dimensional manifold  $\mathcal{M}$ with the structure $(g,\pi,P,\phi)$ that satisfies equation $\overset{g}{\nabla} \pi = \omega (g + \pi\otimes\pi)$ is called \textit{Lorentzian concircular structure manifold} \cite{shaikh2003}. 
	{The manifold that allows a unit torse-forming vector \eqref{eq:torse-formingunit} has the metric of the form (see pp. 193. in \cite{sinyukov1979})
		$$ds^2=\epsilon (dx^1)^2 + f(x^1,x^2,\dots, x^n)d{s^*}^2,$$
		where $f$ is a non-zero function and  $d{s^*}^2 (x^2,\dots, x^n)$ is the metric of an $(n-1)$-dimensional  Riemannian submanifold. For special cases of a  torse-forming vector, we have special cases of the function $f$ (see \cite{mikes2000}) . 
	}
	
	In this paper, we will study a semi-symmetric metric connection whose generator is a concircular vector field in Yano's sense, that is, we will consider the so-called concircularly semi-symmetric metric connection, defined in paper \cite{slosarska1984}. In the paper \cite{mpvz2023}, the properties of curvature tensors with respect to this connection were examined, and tensors independent of its generator were determined. The paper \cite{mzv2025} focuses on the application of this connection to Lorentzian manifolds.


	Generalized-Roberston Walker (GRW) space-time is a special case of Lorentzian manifolds. After B. Y. Chen \cite{chen2014} characterized GRW space-time using a timelike concircular vector (in Fialkow's sense), three years later, C. A. Mantica and L. G. Molinari \cite{manticamolinari2017} characterized  GRW space-time using a unit timelike torse-forming vector field, and in paper \cite{manticamolinari2019} they proved that a Lorentzian concircular structure manifold coincides with GRW space-time. These earlier works provided a new perspective on GRW space-time, particularly through the study of the aforementioned vector fields. Consequently, recently, the study of Lorentzian manifolds with a torse-forming vector field has become very attractive, leading to the discovery of several interesting results (for example, see \cite{ucd2022,blaga2020,siddiqi2022,siddiqi2019}). {Considering that a concircular vector in Yano's sense is a special case of a torse-forming vector, the previous results motivated us to apply the concircularly semi-symmetric metric connection to Lorentzian manifolds. } A Lorentzian manifold with a concircularly semi-symmetric metric connection reduces to GRW space-time, and the mentioned connection reduces to a semi-symmetric metric $P$-connection when its generator is a unit timelike vector \cite{mzv2025}.
	
	Here, we first consider a pseudo-Riemannian manifold equipped with a concircularly semi-symmetric metric connection, and then focus on Lorentzian manifolds, specifically Generalized Robertson–Walker space-times. Motivated by the fact that in a GRW space-time the three curvature tensors, as well as the corresponding Ricci tensors, cannot vanish, in paper \cite{mzv2025} we studied various symmetries of such manifolds with a semi-symmetric metric $P$-connection. Therefore, in this paper, we have considered special manifolds in which the Ricci tensors are proportional to the metric $g$, and in this way we established new conditions for the quasi-Einstein manifold.  We examined the application of this connection to the theory of relativity and found that the equation of state represents a phantom barrier when the perfect fluid space-time satisfies the Einstein field equation with respect to the Levi-Civita connection and with a cosmological constant.
		
		The papers \cite{csillag2024,kranas2019} study the physical properties and application of connections with torsion to the theory of relativity, but unlike this paper, \cite{csillag2024} considers generalized Einstein field equations, and \cite{kranas2019} studies Einstein-Cartan field equations. Given that Lorentzian manifolds with concircularly semi-symmetric metric connection satisfy the equation of a unit torse-forming vector, which is defined only in terms of the Levi-Civita connection, in this paper we have considered Einstein field equations only in relation to the Levi-Civita connection, with the condition that they admit a unit torse-forming vector.

	\section{Concircularly semi-symmetric metric connection}

	Let $(\mathcal{M},g)$ be an $n$-dimensional pseudo-Riemannian manifold $(n>2)$. A \textit{semi-symmetric metric connection} is a connection whose torsion tensor is given by equation
	\begin{equation*}
		\overset{1}{T}(X,Y)= \pi(Y) X - \pi(X)Y,
	\end{equation*} 
	and which satisfies condition
	\begin{equation*}
		\overset{1}{\nabla}g=0.
	\end{equation*}
	This connection is given by equation \cite{pak1969}
	\begin{equation}\label{eq:SSmc}
		\overset{1}{\nabla}_{X} Y = \overset{g}{\nabla}_{X} Y + \pi(Y) X - g (X,Y) P,
	\end{equation}
	where $\pi$ is a covector, $P$ its associated vector field such that $\pi(\cdot)=g(\cdot,P)$ and $\overset{g}{\nabla}$ denotes the Levi-Civita connection.
	
	A \textit{concircularly semi-symmetric metric connection} \cite{slosarska1984} is a special class of a semi-symmetric metric connection whose generator $\pi$ satisfies relation
	\begin{equation}\label{eq:conditionforconcircularmapping}
		(\overset{g}{\nabla}_{X} \pi )(Y) - \pi(X)\pi(Y)=\omega g(X,Y),
	\end{equation}
	where $\omega$ is a scalar function. The previous equation gives
	\begin{equation*}
		\omega = \frac{1}{n}(\mathrm{div}P - \pi(P)).
	\end{equation*}
	
	It is clear that the 1-form $\pi$ satisfying condition (\ref{eq:conditionforconcircularmapping}) is closed. By comparing the equations (\ref{eq:torse-forming}) and (\ref{eq:conditionforconcircularmapping}) we see that they are equal for $\eta=\pi$, which means that the vector satisfying the condition (\ref{eq:conditionforconcircularmapping}) is a special case of a torse-forming vector, i.e. it is a concircular vector in Yano's sense. In \cite{blaga2024} such a vector is called \textit{self-torse-forming}. 

	The covariant derivative of $P$ with respect to a concircularly semi-symmetric metric connection satisfies the following equation 
	\begin{equation}\label{eq:nabla1XP}
		\overset{1}{\nabla}_{X} P = (\omega + \pi(P))X.
	\end{equation}
	
	A semi-symmetric metric connection $\overset{1}{\nabla}$ which satisfies $\overset{1}{\nabla} P=0$ is called a semi-symmetric metric $P$-connection \cite{chaubey2019}. The previous equation shows that a concircularly semi-symmetric metric connection becomes a semi-symmetric metric $P$-connection if and only if $\omega =-g(P,P)$ \cite{mzv2025}. 
	
	From equations (\ref{eq:SSmc}) and (\ref{eq:conditionforconcircularmapping}) we obtain the following statement.
	
	\begin{theorem}\label{thrm21}
		For a concircularly semi-symmetric metric connection the following relations hold
		\begin{enumerate}[label=\textup{(\roman*)}, itemsep=0.5em, leftmargin=*, font=\normalfont]
			\item \(\displaystyle \overset{1}{\nabla}_{P} Y = \overset{g}{\nabla}_{P} Y,\)
			\item \(\displaystyle \pi(\overset{1}{T}(X,Y)) = 0,\)
			\item \(\displaystyle (\overset{g}{\nabla}_{X} \pi)(P) = (\overset{g}{\nabla}_{P} \pi)(X) = \pi(\overset{g}{\nabla}_{P} X) = (\omega + \pi(P)) \pi(X),\)
			\item \(\displaystyle (\mathcal{L}_{P} \pi)(X) = 2(\omega + \pi(P)) \pi(X),\)
			\item \(\displaystyle (\mathcal{L}_{P} g)(X,Y) = 2 (\overset{g}{\nabla}_{X} \pi)(Y).\)
		\end{enumerate}
	\end{theorem}
	
	\begin{proof}
		Using the well-known equation for the Lie derivative of the metric $g$ with respect to the Levi-Civita connection, i.e.
		\begin{equation*}
			(\mathcal{L}_{P} g)(X,Y) = (\overset{g}{\nabla}_{X} \pi)(Y) + (\overset{g}{\nabla}_{Y} \pi)(X),  
		\end{equation*}
		and  using condition (\ref{eq:conditionforconcircularmapping}) we obtain the last identity. The other relations are easy to prove.
		
	\end{proof}

	The Lie derivative of metric $g$ with respect to a non-symmetric connection $\overset{1}{\nabla}$ is defined by equation 
	\begin{align*}
		(\overset{1}{\mathcal{L}}_{P} g) (X,Y) & = Pg(X,Y) - g(\overset{1}{\nabla}_{P}X - \overset{1}{\nabla}_{X} P,Y) - g(X, \overset{1}{\nabla}_{P}Y -  \overset{1}{\nabla}_{Y} P) \\
		& = (\overset{1}{\nabla}_{P} g)(X,Y) + g(\overset{1}{\nabla}_{X} P,Y) + g(X, \overset{1}{\nabla}_{Y} P),
	\end{align*}
	and for a concircularly semi-symmetric metric connection $\overset{1}{\nabla}$ we have
	\begin{equation*}
		(\overset{1}{\mathcal{L}}_{P} g) (X,Y) = 0 + g((\omega + \pi(P))X,Y) + g(X, (\omega + \pi(P))Y),
	\end{equation*}
	where we take into consideration equation (\ref{eq:nabla1XP}). Further, we obtain
	\begin{equation*}
		\overset{1}{\mathcal{L}}_{P} g = 2 (\omega + \pi(P)) g,
	\end{equation*}
	from which the following statements follow.
	\begin{theorem}
		A vector field $P$ is conformal with respect to a concircularly semi-symmetric metric connection.
	\end{theorem}
	\begin{corollary}
		A vector field $P$ is Killing with respect to a concircularly semi-symmetric metric connection if and only if $\omega =-g(P,P)$.
	\end{corollary}

	We initiated  investigation of a concircularly semi-symmetric metric connection in \cite{mpvz2023,mzv2025}, where we examined the properties of six linearly independent curvature tensors, given by the following relations (see e.g. \cite{zlatanovic2021})
	\begin{align*}
		\overset{1}{R} (X,Y)Z = & \overset{1}{\nabla}_{X}\overset{1}{\nabla}_{Y} Z - \overset{1}{\nabla}_{Y}\overset{1}{\nabla}_{X} Z - \overset{1}{\nabla}_{[X,Y]} Z, \\
		\overset{0}{R} (X,Y)Z  = & \overset{1}{R} (X,Y)Z -\frac{1}{2}(\overset{1}{\nabla}_{X} \overset{1}{T})(Y,Z) + \frac{1}{2}(\overset{1}{\nabla}_{Y} \overset{1}{T})(X,Z)  -\frac{1}{4} \underset{XYZ}{\mathfrak{S}}  \overset{1}{T}( \overset{1}{T}(X,Y),Z) - \frac{1}{4}  \overset{1}{T}( \overset{1}{T}(X,Y),Z),
		\\
		\overset{2}{R} (X,Y)Z  = & \overset{1}{R} (X,Y)Z -(\overset{1}{\nabla}_{X} \overset{1}{T})(Y,Z) + (\overset{1}{\nabla}_{Y} \overset{1}{T})(X,Z) - \underset{XYZ}{\mathfrak{S}} \overset{1}{T}( \overset{1}{T}(X,Y),Z), \\ 
		\overset{3}{R} (X,Y)Z  = & \overset{1}{R} (X,Y)Z  + (\overset{1}{\nabla}_{Y} \overset{1}{T})(X,Z), \\
		\overset{4}{R} (X,Y)Z  = & \overset{1}{R} (X,Y)Z  + (\overset{1}{\nabla}_{Y} \overset{1}{T})(X,Z) -   \overset{1}{T}( \overset{1}{T}(X,Y),Z), \\
		\begin{split}
			\overset{5}{R} (X,Y)Z  = & \overset{1}{R} (X,Y)Z -\frac{1}{2}(\overset{1}{\nabla}_{X} \overset{1}{T})(Y,Z) + \frac{1}{2}(\overset{1}{\nabla}_{Y} \overset{1}{T})(X,Z)  -\frac{1}{2} \underset{XYZ}{\mathfrak{S}}  \overset{1}{T}( \overset{1}{T}(X,Y),Z) + \frac{1}{2}  \overset{1}{T}( \overset{1}{T}(Z,X),Y).
		\end{split}
	\end{align*}
	where \(\underset{XYZ}{\mathfrak{S}}\) denotes the cyclic sum of the vectors \(X,Y,Z\). By contracting with respect to vector field $X$ in the previous equations, we obtain the corresponding Ricci tensors $\overset{\theta}{R}ic$, $\theta=0,1,\dots,5$, and $\overset{g}{R}ic$, where $\overset{g}{R}ic$ denotes Ricci tensor with respect to the Levi-Civita connection, and here we will continue with the study of the properties of these Ricci tensors.
	
	\begin{theorem}\cite{mpvz2023}
		Let $(\mathcal{M},g,\overset{1}{\nabla})$ be a pseudo-Riemannian manifold with a concircularly semi-symmetric metric connection.
		The Ricci tensors $\overset{\theta}{R}ic$, $\theta=0,1,2,\ldots,5$, and the Ricci tensor $\overset{g}{R}ic$ are related by equations
		\begin{align}
			\label{eq:Ric0css}
			\overset{0}{R}ic  = &  \overset{g}{R}ic  -\frac{n-1}{2}  (3\omega + \pi(P)) g - \frac{n-1}{4}\pi\otimes\pi,
			\\
			\label{eq:Ric1css}
			\overset{1}{R}ic = & \overset{g}{R}ic - (n-1)(2\omega +\pi(P)) g,
			\\
			\label{eq:Ricalphacss}
			\overset{\alpha}{R}ic = & \overset{g}{R}ic - (n-1)\omega g, \;\; \alpha=2,3,
			\\
			\label{eq:Ric4css}
			\overset{4}{R}ic = & \overset{g}{R}ic - (n-1) \omega  g - (n-1)\pi\otimes\pi,
			\\
			\label{eq:Ric5css}
			\overset{5}{R}ic = & \overset{g}{R}ic- \frac{n-1}{2} (3\omega + \pi(P) )  g - \frac{n-1}{2}\pi\otimes\pi.
		\end{align}
	\end{theorem}
	
	The previous equations show that the Ricci tensors $\overset{\theta}{R}ic$, $\theta=0,1,\dots,5$, are symmetric. 
	If we denote the corresponding scalar curvatures for $\overset{\theta}{R}ic$ and $\overset{g}{R}ic$ by $\overset{\theta}{r}$ and $\overset{g}{r}$, we have the following theorem.
	
	\begin{theorem}
		Let $(\mathcal{M},g,\overset{1}{\nabla})$ be a pseudo-Riemannian manifold with a concircularly semi-symmetric metric connection. The scalar curvatures  $\overset{\theta}{r}$, $\theta=0,1,2,\ldots,5$, and the scalar curvature $\overset{g}{r}$ are related by equations
		\begin{align}
			\label{eq:r0css}
			\overset{0}{r}  = &  \overset{g}{r}   -\frac{3n(n-1)}{2}  \omega - \frac{(n-1)(2n+1)}{4}\pi(P),
			\\
			\label{eq:r1css1}
			\overset{1}{r}  = &  \overset{g}{r}   -2n(n-1) (\omega + \frac{1}{2}\pi(P)),
			\\
			\label{eq:ralphacss}
			\overset{\alpha}{r}  = & \overset{g}{r} - n(n-1)\omega, \;\; \alpha=2,3
			\\
			\label{eq:r4css}
			\overset{4}{r}  = & \overset{g}{r}  - n(n-1) \omega - (n-1)\pi(P),
			\\
			\label{eq:r5css}
			\overset{5}{r}= & \overset{g}{r}- \frac{3n(n-1)}{2}\omega - \frac{n^2-1}{2}\pi(P).
		\end{align}
	\end{theorem}

	\begin{remark}
		If $\overset{1}{\nabla}$ is a metric connection (i.e. $\overset{1}{\nabla}g=0$) with torsion tensor $\overset{1}{T}\ne 0$, then $(\mathcal{M},g,\overset{1}{\nabla})$ is called a Riemann-Cartan manifold \cite{gordevastepanov2010,stepanov2010}, which has applications in theories of gravity, such as the Einstein-Cartan (for example, see \cite{petti2021}).
	\end{remark}
	
	\begin{remark}
		In \cite{yanobochner1954,gordevastepanov2010,stepanov2010,goldberg1956}, a Killing vector with respect to a non-symmetric metric connection is called pseudo-Killing vector.
	\end{remark}

	\begin{remark} A conformal mapping satisfying the condition (\ref{eq:conditionforconcircularmapping}) is called concircular, and such a mapping preserves geodesic circles \cite{yano1940}. More precisely, in paper {\cite{yano1940}} K. Yano defined concircular mappings using the condition
		\begin{equation}\label{eq:concircularcondition}
			(\overset{g}{\nabla}_{X} \pi )(Y) - \pi(X)\pi(Y) + \frac{1}{2}\pi(P) g(X,Y) = \mu g(X,Y),
		\end{equation}
		which is equivalent to (\ref{eq:conditionforconcircularmapping}), where $\mu=\omega + \frac{1}{2}\pi(P)$. A semi-symmetric metric connection with condition (\ref{eq:concircularcondition}) is called $S$-concircular \emph{\cite{stavre1982}}.
	\end{remark}

	\section{Einstein and quasi-Einstein manifolds}\label{sec:quasi-Einstein}
	
	Einstein manifolds represent an important class of (pseudo-)Riemannian manifolds where the Ricci tensor is proportional to the metric, playing a crucial role in  general relativity,  string theory, etc.
	
	A tensor $\overset{g}{E}$ which is given by
	\begin{equation*}
		\overset{g}{E}=	\overset{g}{R}ic- \frac{\overset{g}{r}}{n} g,
	\end{equation*}
	is called 
	\textit{traceless Ricci tensor} 
	or \textit{concircular Ricci tensor} \cite{ucd2005}. If it vanishes then the manifold is \textit{Einstein}. Therefore, an Einstein manifold (with respect to the Levi-Civita connection) is characterized by equation
	\begin{equation*}
		\overset{g}{R}ic= \frac{\overset{g}{r}}{n} g,
	\end{equation*}
	where $\overset{g}{r}$ is constant for $n>2$.
	Einstein type tensors were determined by decomposition of linearly independent curvature tensors in \cite{mazl2022}, while in the paper \cite{petrovic2019} they were determined as invariants for concircular mappings of generalized Riemannian manifolds.
	Since the Ricci tensors $\overset{\theta}{R}ic$ are symmetric, and the corresponding Einstein type tensors $\overset{\theta}{E}$, $\theta =0,1,2,\dots,5$ are also symmetric. Using the Einstein type tensors $\overset{\theta}{E}$, $\theta =0,1,2,\dots,5$,  which are given by the equations
	\begin{equation*}
		\overset{\theta}{E} = \overset{\theta}{R}ic - \frac{\overset{\theta}{r}}{n} g,
	\end{equation*}
	we will define special classes of pseudo-Riemannian manifolds with a concircularly semi-symmetric metric connection.
	
	\begin{definition}
		A pseudo-Riemannian manifold $(\mathcal{M},g,\overset{1}{\nabla})$ with a concircularly semi-symmetric metric connection is called an Einstein type manifold of the $\theta$-th kind, $\theta =0,1,2,\dots,5$, if the Einstein type tensor of the $\theta$-th kind vanishes, i.e. $\overset{\theta}{E}= 0$.
	\end{definition}
	\begin{remark}\label{rem:Einsteintypemanifolds}
		According to the previous definition, the Einstein type manifolds are characterized by equation 
		\begin{equation}\label{eq:Eisteintypemanidolds}
			\overset{\theta}{R}ic = \frac{\overset{\theta}{r}}{n} g.
		\end{equation}
		However, the previous equation does not hold for any non-symmetric connection, because the Ricci tensors $\overset{\theta}{R}ic $, $\theta =0,1,2,\dots,5$, are not symmetric with respect to the non-symmetric connection in the general case. Therefore, in the definition of such manifolds, due to the symmetry of the right-hand side of equality (\ref{eq:Eisteintypemanidolds}), instead of $\overset{\theta}{R}ic $, one must take the symmetric part of these tensors, i.e. $sym\overset{\theta}{R}ic $, where 	\begin{equation*}
			sym\overset{\theta}{R}ic(X,Y)=\frac{1}{2}(\overset{\theta}{R}ic(X,Y)+\overset{\theta}{R}ic(Y,X)).
		\end{equation*}
		As we have already noted, for a concircularly semi-symmetric connection all Ricci tensors $\overset{\theta}{R}ic $, $\theta =0,1,2,\dots,5$, are symmetric, i.e. $\overset{\theta}{R}ic =sym \overset{\theta}{R}ic$ holds, which means that the previous definition is valid.
	\end{remark}

	In the previous paper, we dealt with these manifolds and we proved that Einstein type manifolds of the $\theta$-th kind may, in specific cases, reduce to Einstein manifolds (with respect to the Levi-Civita connection).

	\begin{theorem}\cite{mpvz2023} \label{thm:Einstein123}
		A pseudo-Riemannian manifold $(\mathcal{M},g,\overset{1}{\nabla})$ with a concircularly semi-symmetric metric connection is an Einstein manifold if and only if it is an Einstein type manifold of the first, second or third kind.
	\end{theorem}
	
	The following theorem presents the dependence of tensors  $\overset{\theta}{E}$ and $\overset{g}{E}$, $\theta =0,1,2,\dots,5$.
	
	\begin{theorem}
		In a pseudo-Riemannian manifold $(\mathcal{M},g,\overset{1}{\nabla})$ with a concircularly semi-symmetric metric connection, Einstein type tensors $\overset{\theta}{E}$, $\theta =0,1,2,\dots,5$ satisfy the following relations
		\begin{align}
			\overset{0}{E} & = \overset{g}{E} + \frac{n-1}{4n}\pi(P)g - \frac{n-1}{4}\pi\otimes\pi, \\
			\label{eq:EbetaEg}
			\overset{\beta}{E} & = \overset{g}{E}, \;\; \beta=1,2,3, \\
			\overset{4}{E} & = \overset{g}{E} + \frac{n-1}{n}\pi(P)g - (n-1) \pi\otimes\pi, \\
			\overset{5}{E} & = \overset{g}{E} + \frac{n-1}{2n}\pi(P)g - \frac{n-1}{2}\pi\otimes\pi.
		\end{align}
	\end{theorem}
	From equation (\ref{eq:EbetaEg}) we have a direct consequence.
	\begin{theorem}
		Let $(\mathcal{M},g,\overset{1}{\nabla})$ be a pseudo-Riemannian manifold with a concircularly semi-symmetric metric connection $\overset{1}{\nabla}$ and $\overset{g}{\nabla}$ be the Levi-Civita connection. The traceless Ricci tensor $\overset{g}{E}$ is invariant under connection transformation $\overset{g}{\nabla} \rightarrow \overset{1}{\nabla}$.
	\end{theorem}


	There are many papers dealing with the generalization of Einstein manifolds, with different approaches. For example, quasi-Einstein, nearly quasi-Einstein, mixed quasi-Einstein manifolds generalize the Einstein manifolds. Also, the various types of solitons are generalizations of Einstein manifolds.
	
	A \textit{quasi-Einstein manifold} is a (pseudo-)Riemannian manifold whose Ricci tensor has the form
	\begin{equation}\label{eq:qE}
		\overset{g}{R}ic =ag + b \pi\otimes\pi.
	\end{equation}

	\begin{remark}
		The following cases of quasi-Einstein manifolds have been observed in the literature:
		\begin{enumerate}[label=\textup{(\roman*)}, itemsep=0.5em, leftmargin=*, font=\normalfont]
			\item In \cite{chaki2000}, $a$ and $b$ are scalar functions, and $\pi$ is a 1-form with an associated unit vector;
			\item In \cite{deszczetal998}, $a$ and $b$ are real constants, and $\pi$ is a 1-form;
			\item In \cite{goldberg1980}, $a$ and $b$ are scalar functions, and $\pi$ is a closed 1-form;
			\item In \cite{han2021}, $a$ and $b$ are scalar functions, and $\pi$ is a 1-form.
		\end{enumerate}
	\end{remark}


	In the Theorem \ref{thm:Einstein123} we used Einstein type tensors of the first, second and third kind, and in the following we will use the remaining Einstein type tensors  $\overset{0}{E}$,  $\overset{4}{E}$ and  $\overset{5}{E}$ with respect to a concircularly semi-symmetric metric connection and with their help we will determine new conditions for the quasi-Einstein manifold.
	
	If the manifold is Einstein type of the zeroth kind, i.e. if $\overset{0}{E}=0$ holds, then Ricci tensor of the zeroth kind has the form
	\begin{equation}\label{eq:Einstein0thkind}
		\overset{0}{R}ic =\frac{\overset{0}{r}}{n}g.
	\end{equation}
	
	Aftet substituting equation (\ref{eq:Einstein0thkind}) into (\ref{eq:Ric0css}), we get
	\begin{equation*}
		\begin{split}
			\frac{1}{4}(4\overset{g}{r}   -6n(n-1)\omega - (n-1)(2n+1)\pi(P)) g = \overset{g}{R}ic  & - \frac{n-1}{2}  (3\omega + \pi(P)) g  - \frac{n-1}{4}\pi\otimes\pi,
		\end{split}
	\end{equation*}
	where we take into account equation (\ref{eq:r0css}). From the previous equation it follows
	\begin{equation*}
		\overset{g}{R}ic = \frac{1}{4n}(4\overset{g}{r}  - (n-1)\pi(P)) g + \frac{n-1}{4}\pi\otimes\pi,
	\end{equation*}
	from where we can conclude that the manifold is quasi-Einstein (according to \cite{goldberg1980}).
	
	Conversely, if we assume that the previous equation holds, then based on the equations (\ref{eq:Ric0css}) and (\ref{eq:r0css}), after a simple calculation, we get
	\begin{equation*}
		\overset{0}{R}ic= \frac{\overset{0}{r}}{n}g,
	\end{equation*}
	which shows that such a manifold is Einstein type of the zeroth kind.
	\begin{theorem}\label{thm:Ajns0}
		A pseudo-Riemannian manifold $(\mathcal{M},g,\overset{1}{\nabla})$ with a concircularly semi-symmetric metric connection is an Einstein type manifold of the zeroth kind if and only it is a quasi-Einstein manifold whose Ricci tensor $\overset{g}{R}ic$ satisfies the relation
		\begin{equation*}
			\overset{g}{R}ic = \frac{1}{4n}(4\overset{g}{r}  - (n-1)\pi(P)) g + \frac{n-1}{4}\pi\otimes\pi.
		\end{equation*}
	\end{theorem}
	
	Analogously, using Ricci tensors and curvature scalars of the fourth and fifth kind, we can prove the following theorems.
	\begin{theorem}
		A pseudo-Riemannian manifold $(\mathcal{M},g,\overset{1}{\nabla})$ with a concircularly semi-symmetric metric connection is an Einstein type manifold of the fourth kind if and only it is a quasi-Einstein manifold whose Ricci tensor $\overset{g}{R}ic$ satisfies the relation
		\begin{equation*}
			\overset{g}{R}ic = \frac{1}{n}(\overset{g}{r}  - (n-1)\pi(P)) g + (n-1)\pi\otimes\pi.
		\end{equation*}
	\end{theorem}
	\begin{theorem}
		A pseudo-Riemannian manifold $(\mathcal{M},g,\overset{1}{\nabla})$ with a concircularly semi-symmetric metric connection is an Einstein type manifold of the fifth kind if and only it is a quasi-Einstein manifold whose Ricci tensor $\overset{g}{R}ic$ satisfies the relation
		\begin{equation*}
			\overset{g}{R}ic = \frac{1}{2n}(2\overset{g}{r}  - (n-1)\pi(P)) g + \frac{n-1}{2}\pi\otimes\pi.
		\end{equation*}
	\end{theorem}

	Since a perfect fluid space-time is example of quasi-Einstein manifolds, it is natural to try to apply the previous results to a Lorentzian manifold and perfect fluid space-time.

	

	\section{Lorentzian manifolds}
	
	A Lorentzian manifold represents a class of pseudo-Riemannian manifolds with a Lorentzian metric $g$ of signature $(1,n-1)$ (or, equivalently, $(n-1,1)$). This manifold is important in applications of theory of relativity and cosmology, because space-time is a four-dimensional time-oriented Lorentzian manifold.  Generalized Robertson-Walker (briefly, GRW) space-time is a special class of Lorentzian manifolds and it is defined in \cite{alias1995}. 
	
	\begin{definition}
		An $n$-dimensional Lorentzian manifold ($n\geq3$) is GRW space-time if the metric $g$ has the form
		\begin{equation*}
			ds^2=g_{ij}dx^i dx^j = -(dt)^2 + f(t)^2 g^{*}_{\mu\nu}(\vec{x}) dx^{\mu}dx^{\nu},
		\end{equation*}
		where $t$ is time and $g^{*}$ is the metric of a Riemannian submanifold (of dimension $(n-1)$).
	\end{definition}
	
	If the metric $g^{*}$ has dimension $3$ and constant curvature then a GRW space-time is actually a Robertson-Walker (briefly, RW) space-time. Thus, GRW space-times extend a RW space-time and, in addition, includes some other space-times, such as Lorentz-Minkowski, Einstein-de Sitter, de Sitter, Friedman cosmological model, etc.
	
	In the last several years, many authors have observed Lorentzian manifolds with various types of semi-symmetric connections, which can be found in the following papers \cite{li2023Ln,ucd2024,yilmaz2023,chaubeysuhde2020, chaubey2021, li2024Ln,chaubey2022b,suh2024}. Here we will study concircularly semi-symmetric metric connection on Lorentzian manifolds.
	
	Let the generator of a concircularly semi-symmetric metric connection be a unit timelike vector, i.e. $g(P,P)=-1$. In the paper \cite{mzv2025} it was proved that for such a connection on Lorentzian manifolds one obtains $\omega=1$, and directly based on (\ref{eq:conditionforconcircularmapping}) the following relation holds  
	\begin{equation*}
		(\overset{g}{\nabla}_{X} \pi )(Y) = g(X,Y) + \pi(X)\pi(Y),
	\end{equation*}
	from where
	\begin{equation}\label{eq:Ptorse-forming1}
		\overset{g}{\nabla}_{X} P= X + \pi(X)P,
	\end{equation}
	i.e.
	\begin{equation}\label{eq:Ptorse-forming}
		\overset{g}{\nabla}_{X} P= -\pi(P)(X + \pi(X)P).
	\end{equation}
	Also, the following two theorems hold.
	
	\begin{theorem}\cite{mzv2025}\label{thm:ssmPcon}
		If a vector field $P$ is a unit timelike vector, then a concircularly semi-symmetric metric connection is a semi-symmetric metric $P$-connection.
	\end{theorem}
	
	\begin{theorem}\cite{mzv2025}\label{cor:CSSGRW}
		An $n$-dimensional Lorentzian manifold $(n\geq3)$ equipped with a concircularly semi-symmetric metric connection whose associated vector $P$ is a unit timelike vector field is a GRW space-time.
	\end{theorem}

	Hence, the triple $(\mathcal{M},g,\overset{1}{\nabla})$ represents an $n$-dimensional GRW space-time with semi-symmetric metric $P$-connection.
	\begin{theorem}
		In a GRW space-time $(\mathcal{M},g,\overset{1}{\nabla})$ the following relations hold
		
		\begin{enumerate}[label=\textup{(\roman*)}, itemsep=0.5em, leftmargin=*, font=\normalfont]
			\item $(\overset{g}{\nabla}_{X} \pi )(P)  =  (\overset{g}{\nabla}_{P} \pi )(X) = \pi (\overset{g}{\nabla}_{P}X) = 0$,
			\item \(\displaystyle ({\mathcal{L}}_{P} \pi) (X)  = 0,  \)
			\item \(\displaystyle 	\overset{1}{T} (P,X)  = \overset{g}{\nabla}_X P. \)
		\end{enumerate}
		
	\end{theorem}
	\begin{proof}
		Considering that $g(P,P)=-\omega=-1$ holds in a GRW space-time $(\mathcal{M},g,\overset{1}{\nabla})$, the previous relations are a direct consequence of Theorem \ref{thrm21}.   
	\end{proof}

	From the equation (\ref{eq:Ptorse-forming}) we can easily conclude that $\overset{g}{\nabla}_P P=0$, which means that the vector $P$ is geodesic.
	The following assertion is also easily proven.
	\begin{theorem}
		A torsion tensor of a semi-symmetric metric $P$-connection $	\overset{1}{\nabla}$ is parallel with respect to $\overset{1}{\nabla}$.
	\end{theorem}


	In the following, we will consider the Ricci tensors $\overset{\theta}{R}ic$, $\theta=0,1,2,\ldots,5$, which will give us the conditions for a GRW space-time to be a perfect fluid space-time. Since we now have values for $\pi(P)$ and $\omega$, the equations of the Ricci tensors (\ref{eq:Ric0css})-(\ref{eq:Ric5css}) get a simpler form, which we will present in the next theorem.

	\begin{theorem}
		In a GRW space-time $(\mathcal{M},g,\overset{1}{\nabla})$, Ricci tensors $\overset{\theta}{R}ic$, $\theta=0,1,2,\ldots,5$, and the Ricci tensor $\overset{g}{R}ic$ are related by equations
		\begin{align*}
			\overset{0}{R}ic  = &  \overset{g}{R}ic  -\frac{n-1}{4} (4g + \pi\otimes\pi ) ,
			\\
			\overset{\beta}{R}ic = & \overset{g}{R}ic - (n-1) g, \;\; \beta=1,2,3,
			\\
			\overset{4}{R}ic = & \overset{g}{R}ic - (n-1)  (g + \pi\otimes\pi ),
			\\
			\overset{5}{R}ic = & \overset{g}{R}ic- \frac{n-1}{2} (2g + \pi\otimes\pi ).
		\end{align*}
	\end{theorem}
	
	The previous equations show us that the tensors $\overset{1}{R}ic$, $\overset{2}{R}ic$ and $\overset{3}{R}ic$ coincide, so in the following, instead of $\overset{\beta}{R}ic$, $\beta=1,2,3$, we will use only $\overset{1}{R}ic$.
	
	Based on the property of the parallelism of the vector $P$ with respect to the connection $\overset{1}{\nabla}$ (according to Theorem \ref{thm:ssmPcon}), it is easy to show that in a GRW space-time $(\mathcal{M},g,\overset{1}{\nabla})$ holds \cite{yilmaz2023}
	\begin{align*}
		\overset{1}{R}ic (P,X) & =0.
	\end{align*}
	
	On the other hand, the Ricci tensor $\overset{g}{R}ic$, in a GRW space-time $(\mathcal{M},g,\overset{1}{\nabla})$, has the following property (see \cite{chaubeysuhde2020,siddiqi2019})
	\begin{align}\label{eq:RicPX}
		\overset{g}{R}ic(P,X) & = (n-1)\pi(X),
	\end{align}
	while other Ricci tensors satisfy the following relations \cite{mzv2025}
	\begin{equation*}
		4\overset{0}{R}ic(P,X) = \overset{4}{R}ic(P,X) = 2\overset{5}{R}ic(P,X) = (n-1)\pi(X).
	\end{equation*}
	which shows that $\frac{n-1}{4}$, $(n-1)$ and $\frac{n-1}{2}$ are eigenvalues of the Ricci tensors $\overset{0}{R}ic$, $\overset{4}{R}ic$, $\overset{5}{R}ic$, respectively, corresponding to the eigenvector $P$. 
	On the other hand, the equation (\ref{eq:RicPX}) shows that $(n-1)$ is an eigenvalue of the Ricci tensor $\overset{g}{R}ic$ corresponding to the eigenvector $P$. In the paper \cite{mzv2025} it was proved that none of the Ricci tensors $\overset{0}{R}ic$, $\overset{4}{R}ic$, $\overset{5}{R}ic$, can be equal to zero.

	{
		Based on equation \eqref{eq:RicPX} the following statement can be proven.
		\begin{theorem}[Lemma 3 in \cite{li2023Ln}]
			On every $n$-dimensional GRW space-time $(\mathcal{M},g,\overset{1}{\nabla})$ holds
			\begin{equation}\label{eq:rGRW}
				P(\overset{g}{r})=2(n(n-1) - \overset{g}{r}),
			\end{equation}
			where $\overset{g}{r}$ denotes scalar curvature.
		\end{theorem}
	}

	\section{Perfect fluid space-time}
	
	
	If the Ricci tensor $\overset{g}{R}ic$ has the form
	\begin{equation}\label{eq:PF}
		\overset{g}{R}ic=ag+b\pi\otimes\pi,
	\end{equation}
	then a Lorentzian manifold is said to be a \textit{perfect fluid space-time}, where $a$ and $b$ are scalars. As we have already mentioned, a perfect fluid space-time is a special case of quasi-Einstein manifolds, because quasi-Einstein manifolds have a metric of arbitrary signature. Papers \cite{neill1983,dede2023,dede2024,gutierrez2009} give results on the relation of a perfect fluid space-time with RW and GRW space-time.
	
	Contraction of the previous equation gives 
	$$ \overset{g}{r} = a\,n - b.$$
	
	Based on the equation (\ref{eq:RicPX}) and (\ref{eq:PF}), it is easy to check that in a  perfect fluid space-time with a semi-symmetric metric $P$-connection, the relation (see Lemma 5. in \cite{li2023Ln}) 
	\begin{equation}\label{eq:a-b}
		a-b=n-1,
	\end{equation}
	holds, and by combining the previous two equations we get
	\begin{equation}\label{eqab}
		a=\frac{\overset{g}{r}}{n-1} - 1, \;\;\;  b= \frac{\overset{g}{r}}{n-1} - n,
	\end{equation}
	which means that the Ricci tensor $\overset{g}{R}ic$ can be written in the form
	\begin{equation*}
		\overset{g}{R}ic = \left(\frac{\overset{g}{r}}{n-1} - 1\right) g + \left(\frac{\overset{g}{r}}{n-1} - n \right)\pi\otimes\pi.
	\end{equation*}

{
		The equation \eqref{eqab} implies the following statement.
		
	\begin{theorem}\label{thm:scalarsPF}
			The scalar curvature of a perfect fluid space-time $(\mathcal{M},g,\overset{1}{\nabla})$  is constant if and only if any of the scalars $a$ or $b$ is constant.
		\end{theorem}

Differentiating equation \eqref{eq:PF} and using \eqref{eq:Ptorse-forming1}, gives
\begin{equation*}
	(\nabla_X \overset{g}{Q})Y = \frac{X(\overset{g}{r})}{n-1} (Y+\pi(Y)P)+ \big(\frac{\overset{g}{r}}{n-1} -n\big) (g(X,Y)P + 2\pi(X)\pi(Y)P + \pi(Y)X).
\end{equation*}
By contracting the previous equation with respect to $X$, we have
\begin{equation*}
	\frac{n-3}{2(n-1)}Y(\overset{g}{r}) = \frac{1}{n-1} (P(\overset{g}{r})+(n-1)(\overset{g}{r}-n(n-1)))\pi(Y),
\end{equation*}
where we used identity $2(\mathrm{div}Q)Y=Y(\overset{g}{r})$.  From the previous equation and \eqref{eq:rGRW} we get
\begin{equation*}
	Y(\overset{g}{r})=2(\overset{g}{r}-n(n-1))\eta(Y),
\end{equation*}
which proves the following theorem.
\begin{theorem}\label{thm:scalarcurvaturePF}
	The scalar curvature of a perfect fluid space-time $(\mathcal{M},g,\overset{1}{\nabla})$  is constant if and only if it has the form $\overset{g}{r}=n(n-1)$.
\end{theorem}
Theorems \ref{thm:scalarsPF} and \ref{thm:scalarcurvaturePF} imply the following statement.
\begin{corollary}\label{cor:mainPF}
	If any of the scalars $a$ or $b$ is constant then a perfect fluid space-time $(\mathcal{M},g,\overset{1}{\nabla})$ reduces to Einstein manifold.
\end{corollary}
\begin{proof}
	Indeed, if either of the scalars $a$ or $b$ is constant, then by Theorems \ref{thm:scalarsPF} and \ref{thm:scalarcurvaturePF} the scalar curvature $\overset{g}{r}$ is also constant of the form $\overset{g}{r}=n(n-1)$. Substituting this value into \eqref{eqab}, we obtain
	\begin{equation}\label{eq:mainconstPF}
		a=n-1, \quad b=0,
	\end{equation}
	which means that the Ricci tensor has the form $\overset{g}{R}ic=(n-1)g$, thus completing the proof.
\end{proof}

	}

	Since the Ricci tensors $\overset{0}{R}ic$, $\overset{4}{R}ic$, $\overset{5}{R}ic$ cannot vanish identically (see Theorem 4.5 in \cite{mzv2025}), we consider weaker conditions than the vanishing conditions. This leads to the definition of special classes of GRW space-times $(\mathcal{M},g,\overset{1}{\nabla})$ (see Definition 5.1 in \cite{mzv2025}).
	\begin{definition}
		A GRW space-time $(\mathcal{M},g,\overset{1}{\nabla})$ is a perfect fluid space-time of the $\theta$-th kind, $\theta=0,1,4,5$ if
		\begin{equation*}
			\overset{\theta}{R}ic=\overset{\theta}{a}g+\overset{\theta}{b}\pi\otimes\pi, \;\;\; \theta=0,1,4,5,
		\end{equation*}
		where $\overset{\theta}{a}$, $\overset{\theta}{b}$ are smooth functions.
	\end{definition}
	\begin{remark}
		Similar to Remark \ref{rem:Einsteintypemanifolds}, the previous definition is valid due to the symmetry of the Ricci tensors $\overset{\theta}{R}ic $ with respect to a  semi-symmetric metric $P$-connection. In the general case, for any non-symmetric connection one must take $sym\overset{\theta}{R}ic $ instead of $\overset{\theta}{R}ic $, due to the symmetry of the right-hand side of the previous equation. 
	\end{remark}
	
	By using the properties of Ricci tensors $\overset{\theta}{R}ic$, the following theorem can be proved.
	
	\begin{theorem}\cite{mzv2025}
		A GRW space-time $(\mathcal{M},g,\overset{1}{\nabla})$ is a perfect fluid space-time if and only if it is a perfect fluid space-time of the $\theta$-th kind, $\theta=0,1,4,5$.
	\end{theorem}

{
	
	In contrast to the previous definition, in Section \ref{sec:quasi-Einstein} we set slightly stricter conditions for the Ricci tensors $\overset{0}{R}ic$,$\overset{4}{R}ic$, $\overset{5}{R}ic$, i.e. we observed the case when they are proportional to the metric $g$, and we defined Einstein type manifold of the $\theta$-th kind. A GRW space-time $(\mathcal{M},g,\overset{1}{\nabla})$ of Einstein type of the $\theta$-th kind, $\theta=0,4,5$, is actually a special case of a perfect fluid space-time of the  $\theta$-th kind, for
	\begin{equation*}
		\overset{\theta}{a}=\frac{\overset{\theta}{r}}{n} \; \;\; \mbox{and} \; \; \; \overset{\theta}{b}=0.
	\end{equation*}

For example, if the manifold is Einstein type of the zeroth kind, then from (\ref{thm:Ajns0}) we have that a perfect fluid space-time has the Ricci tensor of the form
\begin{equation}\label{eq:RicAjn02}
	\overset{g}{R}ic  = \frac{1}{4n} (4\overset{g}{r} + n-1)g  + \frac{n-1}{4} \pi\otimes\pi.
\end{equation}
Further, taking into consideration equations (\ref{eq:PF}) and (\ref{eq:mainconstPF}), we get relation
\begin{equation*}
 \frac{n-1}{4} = 0,
\end{equation*}
which is in contradiction with $n\geq3$, and as a result we can formulate the following theorem.
	\begin{theorem}\label{thmcannotEtype}
	An $n$-dimensional $(n\geq3)$ GRW space-time $(\mathcal{M},g,\overset{1}{\nabla})$ cannot be an Einstein type of the $\theta$-th kind, $\theta=0,4,5$.
\end{theorem}
}

	\section{Theory of relativity}
	
	In order to apply the previous results to the theory of relativity, we will consider Einstein field equations with cosmological constant $\Lambda$ which read
	\begin{equation}\label{eq:EFEL}
		\overset{g}{R}ic - \frac{\overset{g}{r}}{2} g + \Lambda g = k\tau,
	\end{equation}
	where $\tau$ is \textit{energy-momentum tensor} (of type $(0,2)$) and $k$ is \textit{gravitational constant}. A energy-momentum tensor of a perfect fluid space-time has the form
	\begin{equation}\label{eq:energy-momentum}
		\tau = \rho g + (\sigma+p) \pi \otimes \pi,
	\end{equation}
	where $\sigma$ is the \textit{energy density} and $p$ is the \textit{isotropic pressure}, where $\sigma+p\ne 0$ and $\sigma>0$ (see pp. 61-63 in \cite{duggal1999} or pp. 61. in \cite{stephani2009}). 
	

	
	Given that the divergence of the left hand side of the equality (\ref{eq:EFEL}) is equal to zero, the same should be true for the right hand side of that equality. For the divergence of the energy-momentum tensor (\ref{eq:energy-momentum}) of a perfect fluid space-time, we have
	\begin{equation}\label{eq:divtau}
		\mathrm{div}\tau = (\sigma+p)\mathrm{div}\pi\otimes\pi.
	\end{equation}
	
	However, taking into account the equation (\ref{eq:Ptorse-forming}), which is valid for a semi-symmetric metric $P$-connection, we further have
	\begin{equation*}
		\mathrm{div} \pi\otimes\pi = (n-1)\pi,
	\end{equation*}
	which means that the equation (\ref{eq:divtau}) is equal to zero if and only if $\sigma+p=0$, from where follows the equation of state $\frac{p}{\sigma}=-1$, which is the limit value for \textit{phantom dark energy}.
	\begin{theorem}
		In a perfect fluid space-time $(\mathcal{M},g,\overset{1}{\nabla})$ which satisfies the Einstein field equations with the cosmological constant, the equation of state represents a phantom barrier.
	\end{theorem}
	
	This theorem can be written in the following equivalent form as well.
	\begin{theorem}
		Let a perfect fluid space-time with a unit timelike torse-forming vector field	 (\ref{eq:Ptorse-forming1}) satisfies the Einstein field equations with the cosmological constant. Then the equation of state represents a phantom barrier.
	\end{theorem}

	\section{Conclusions}
	
	In the paper, we continue with the research of concircularly semi-symmetric metric connection and examine the properties of its generator. After proving that the strong energy condition is violated in a perfect fluid space-time with a semi-symmetric metric $P$-connection that satisfies the Einstein field equation without a cosmological constant \cite{mzv2025}, in the paper we have shown that the equation of state of a perfect fluid space-time with a semi-symmetric metric $P$-connection that satisfies the Einstein field equation with a cosmological constant represents a phantom barrier. It will be interesting to examine the application of the mentioned connections to various other modifications of the theory of relativity, such as the generalized Einstein field equations (e.g. \cite{csillag2024}), which would extend the application of the observed connection to the theory of relativity.

	Theorem \ref{thmcannotEtype} is a correction of the results from section 2.3.3 in the first author's PhD thesis \cite{mPhD2025}.
	
	
	\vskip0.5cm
	
	\section*{Statements and Declarations}

\noindent {\bf Funding} This work was partially supported by the Ministry of Education, Science and Technological Development of the Republic of Serbia (contract reg. no. 451-03-65/2025-03/200123 for Miroslav D. Maksimovi\' c and Marija S. Najdanovi\' c, and contract reg. no. 451-03-137/2025-03/200124  for Milan Lj. Zlatanovi\' c), as well as by the Bulgarian Ministry of Education and Science under the Scientific Programme "Enhancing the Research Capacity in Mathematical Sciences (PIKOM)", No. DO1-67/05.05.2022 (for Milan Lj. Zlatanovi\' c and Miroslav D. Maksimovi\' c).

\vskip0.5cm

\noindent	{\bf Author contributions} All authors contributed equally. They read and approved the final manuscript.

\vskip0.5cm
\noindent	{\bf Declarations}
\vskip0.5cm
\noindent {\bf Conflicts of Interest}	 The author declares no conflicts of interest.

\vskip0.5cm

\noindent {\bf Data Availability Statement}
The authors declare that no associated data is
included.

	Miroslav D. Maksimovi\'c \\
	{University of Pri\v stina in Kosovska Mitrovica, Faculty of Sciences and Mathematics, Department of Mathematics, Kosovska Mitrovica, Serbia, and \\
		Institute of Mathematics and Informatics, Bulgarian Academy of Sciences, Sofia 1113, Acad. G. Bonchev Str., Bl. 8, Bulgaria}
	\\
	email: miroslav.maksimovic@pr.ac.rs
	\\

	Milan Lj. Zlatanovi\'c \\
	{University of Ni\v s, Faculty of Sciences and Mathematics, Department of Mathematics, Ni\v s, Serbia,}
	\\
	email: zlatmilan@yahoo.com
	\\
	
	Marija S. Najdanovi\'c \\
	{University of Pri\v stina in Kosovska Mitrovica, Faculty of Sciences and Mathematics, Department of Mathematics, Kosovska Mitrovica, Serbia,}
	\\
	email: marija.najdanovic@pr.ac.rs
	\\

\end{document}